\documentclass[11pt]{article}

\usepackage[margin=1in]{geometry}  
\usepackage{graphicx}              
\usepackage{amsmath}               
\usepackage{amsfonts}              
\usepackage{amsthm}                

\usepackage{amssymb}
\usepackage{amscd} 
\usepackage{epsfig}

\newtheorem{thm}{Theorem}[section]
\newtheorem{lem}[thm]{Lemma}
\newtheorem{prop}[thm]{Proposition}
\newtheorem{cor}[thm]{Corollary}

\newtheorem{hyp}{Hypothesis}
\newtheorem{rem}{Remark}

\newcommand{\RR}{\mathbb{R}}      

\numberwithin{equation}{section}

\begin{document}

\nocite{*}

\title{Small mass asymptotics of a charged particle in a variable magnetic field}

\author{Jong Jun Lee  \\
Department of Mathematics\\
University of Maryland\\
College Park, Maryland 20742-4015}

\date{}

\maketitle

\begin{abstract}
We consider small mass asymptotics of the motion of a charged particle in a noisy force field combined with a variable magnetic field. The Smoluchowski-Kramers approximation does not hold in this case. We show that after a regularization of noise, a Smoluchowski-Kramers type approximation works.
\end{abstract}

\section{Introduction}

Consider a charged particle of mass $\mu>0$ moving on a plane. Let the position of this particle at time $t$ be $q^{\mu}_t \in \mathbb{R}^2$. We may express the force field with random noise on the plane as
\begin{equation*}
b(q^{\mu}_t)+\sigma(q^{\mu}_t)\dot{w}_t,
\end{equation*}
where $b : \RR^2 \rightarrow \RR^2$ is a vector-valued function, $\sigma : \RR^2 \rightarrow M_2(\RR)$ is a matrix-valued function, and $w_t \in \RR^2$ is a two dimensional Wiener process. 

Now, suppose that the motion of the particle is subject to a variable magnetic field perpendicular to the plane. The force on the particle due to this magnetic field can be expressed as
\begin{equation}\label{neq01}
A(q^{\mu}_t) \,\dot{q}^{\mu}_t=\alpha(q^{\mu}_t)A_0\,\dot{q}^{\mu}_t,
\end{equation} 
where $\alpha : \RR^2 \rightarrow \RR^+$ is a positive real-valued function and
\begin{equation*}
A_0 :=\left(\begin{array}{ccc} 0 & -1 \\
1 & 0 \end{array} \right).
\end{equation*}

The motion of this particle is governed by the Newton law, so that 
\begin{equation}\label{nneq01}
\left\{ \begin{array}{ll} \mu \, \ddot{q}^{\mu}_t
=b
(q^{\mu}_t)+A(q^{\mu}_t) \dot{q}^{\mu}_t+\sigma(q^{\mu}_t)
\, \dot{w}_t
\\\\
q^{\mu}_0=q_0  \in \mathbb{R}^2 ,\quad
\dot{q}^{\mu}_0=p_0 \in \mathbb{R}^2.
\end{array} \right.
\end{equation}

Notice that equation~\eqref{nneq01} can be rewritten as a system of first order equations
\begin{equation}\label{eq1}
\left\{ \begin{array}{ll}
\dot{q}^{\mu}_t=p^{\mu}_t
\\\\
\dot{p}^{\mu}_t =\frac{1}{\mu} b
(q^{\mu}_t)+\frac{1}{\mu} A(q^{\mu}_t)
\, p^{\mu}_t+\frac{1}{\mu} \sigma(q^{\mu}_t) \,
\dot{w}_t
\\\\
q^{\mu}_0=q_0  \in \mathbb{R}^2 ,\quad p^{\mu}_0=p_0 \in \mathbb{R}^2.
\end{array} \right.
\end{equation}

Now, let  $q_t$ be the solution of the following first order SDE with $\mu=0$ from equation~\eqref{nneq01}:
\begin{equation*}
\left\{ \begin{array}{ll} \dot{q}_t=-A^{-1}(q_t)b (q_t)-A^{-1}(q_t)\sigma(q_t) \dot{w}_t
\\\\
q_0 \in \mathbb{R}^2.
\end{array} \right.
\end{equation*}

It is natural to consider the convergence of $q^{\mu}_t$ to $q_t$ as $\mu \downarrow 0$. The Smoluchowski-Kramers approximation tells us that in the case of the Langevin equation, that is, in the case of
\begin{equation*}
A(q)\equiv-cI
\end{equation*}
for a constant $c>0$, we may replace $q^{\mu}_t$ with $q_t$ for small $\mu$ due to the convergence
\begin{equation*}
\lim_{\mu \downarrow 0} E \max_{0 \leq t
\leq T}|q^{\mu}_t-q_t| = 0
\end{equation*}
(\cite{F01}). However, this convergence is not satisfied for general $A(q)$. Even in the case of $A(q)$ being a constant matrix, if the real parts of the eigenvalues of $A(q)$ are nonnegative, the Smoluchowski-Kramers approximation does not hold. For example, for $A(q)\equiv A_0$, the above convegence fails due to 
\begin{equation*}
\lim_{\mu \downarrow 0} \int_0^t \sin \left(\frac{s}{\mu}\right)dw_s \neq 0
\end{equation*} 
(\cite{CF01}). So we cannot use this approximation in our case. Nonetheless, we may regularize the problem and check a convergence similar to the Smoluchowski-Kramers approximation (\cite{CF01}, \cite{FH01}).

Firstly, it is physically reasonable to introduce a small friction proportional to the velocity. We may write $A_{\epsilon}(q)=A(q)-\epsilon \,I$ and approximate $q^{\mu}_t$ with $q^{\mu,\epsilon}_t$, the solution of the following SDE
\begin{equation*}
\left\{ \begin{array}{ll}
\dot{q}^{\mu,\epsilon}_t=p^{\mu,\epsilon}_t
\\\\
\dot{p}^{\mu,\epsilon}_t =\frac{1}{\mu} b
(q^{\mu,\epsilon}_t)+\frac{1}{\mu} A_{\epsilon}(q^{\mu,\epsilon}_t)
\, p^{\mu,\epsilon}_t+\frac{1}{\mu} \sigma(q^{\mu,\epsilon}_t) \,
\dot{w}_t .
\end{array} \right.
\end{equation*}

This small friction term makes the real parts of the eigenvalues of $A_{\epsilon}(q)$ negative and gives us an exponential decay of the term 
\begin{equation*}
\frac{1}{\mu}\,\exp \left(\frac{1}{\mu}\int_0^t \,A_{\epsilon}(q^{\mu,\epsilon}_s)\, ds\right)=\frac{1}{\mu}\exp \left (-\frac{\epsilon}{\mu}t\right)\left(\begin{array}{ccc}
\cos \left(\frac{1}{\mu}\int_0^t\alpha(q^{\mu, \epsilon}_s)ds\right) & -\sin \left(\frac{1}{\mu}\int_0^t\alpha(q^{\mu, \epsilon}_s)ds\right) \\
\sin \left(\frac{1}{\mu}\int_0^t\alpha(q^{\mu, \epsilon}_s)ds\right) & \cos \left(\frac{1}{\mu}\int_0^t\alpha(q^{\mu, \epsilon}_s)ds\right)
\end{array} \right)
\end{equation*}
as $\mu \downarrow 0$. However, it turns out that this approximation does not support us with enough regularity for the convergence of the system. This follows from 
\[
\lim_{\mu \downarrow 0} \int_0^t \frac{1}{\mu}
\exp \left (-\frac{2\epsilon}{\mu}s\right)\left|\int_0^s
\exp \left (\frac{\epsilon}{\mu}r\right) dw_r\right|^2
ds \neq 0.
\]

As another regularization method, we may approximate the Wiener process $w_t$ with a $\delta$-correlated smooth process $w^{\delta}_t$ as in \cite[Example 7.3 Chapter VI]{IW01}. In section 2, we will prove that as $\mu \downarrow 0$ and $\delta \downarrow 0$ in the way that  $\mu e^{\frac{C}{\delta^2}} \downarrow 0$ for each constant $C>0$ , the solution $q^{\mu,\delta}_t$ of approximated second order equation \eqref{eq2} converges to the solution $\hat{q}_t$ of first order SDE \eqref{neq04} in the sense that
\begin{equation*}
\lim_{\mu \downarrow 0, \delta \downarrow 0, \mu e^{\frac{C}{\delta^2}} \downarrow 0}E \max_{0 \leq t
\leq T}\left|q^{\mu,\delta}_t-\hat{q}_t\right| = 0.
\end{equation*}
In section 3, we consider an application of this approximation, a homogenization problem.

Throughout the present paper we shall use $|\cdot|$ as the standard Euclidean norm in $\mathbb{R}^n$ and $|\cdot|_{\infty}$ as the supremum norm in spaces of functions. Moreover, symbols $C$ and $C_i$'s will indicate arbitrary large positive constants. $C$ and $C_i$'s may take different values in different places.

\section{Small mass asymptotics under the regularized Wiener process}

We define $w^{\delta}_t$ as a mollification of the Wiener process $w_t$ as in \cite [Example 7.3 Chapter VI]{IW01}.

\begin{equation}\label{eq4}
w^{\delta}_t=\frac{1}{\delta}\int_0^{\infty} w(s)\rho(\frac{s-t}{\delta})ds,
\end{equation}
where $\rho(t) \geq 0$ is smooth, has the support in $[0,1]$, and satisfies
\begin{equation*}
\int_0^1 \rho(s) ds =1.
\end{equation*}

$w^{\delta}_t$ is a smooth approximation of $w_t$ satisfying
\begin{equation*}
\lim_{\delta \downarrow 0} E \max_{0 \leq t \leq T}|w^{\delta}_t-w_t|^2 = 0.
\end{equation*}

Now, we rewrite \eqref{eq1} with $w^{\delta}_t$ in place of $w_t$
\begin{equation}\label{eq2}
\left\{ \begin{array}{ll}
\dot{q}^{\mu,\delta}_t=p^{\mu,\delta}_t
\\\\
\mu \, \dot{p}^{\mu,\delta}_t =b
(q^{\mu,\delta}_t)+A(q^{\mu,\delta}_t)
\, p^{\mu,\delta}_t+\sigma(q^{\mu,\delta}_t) \,
\dot{w}^{\delta}_t
\\\\
q^{\mu}_0=q_0  \in \mathbb{R}^2 ,\quad p^{\mu}_0=p_0 \in \mathbb{R}^2.
\end{array} \right.
\end{equation}

To give enough regularity for the problem, we assume the following conditions on $b(q)$, $\sigma(q)$, and $\alpha(q)$.

\begin{hyp}\label{hyp01}\ 
\begin{enumerate}
\item $b : \RR^2 \rightarrow \RR^2$ and $\sigma : \RR^2 \rightarrow M_2(\RR)$ are differentiable and bounded with their derivatives.
\item $\alpha : \RR^2 \rightarrow \RR$ is differentiable and bounded with its derivative. Moreover, 
\begin{equation*}
\inf_{q \in \RR^2} \alpha(q) = \alpha_0>0.
\end{equation*}
\end{enumerate}
\end{hyp}

Under Hypothesis~\ref{hyp01}, we have the relation
\begin{equation*}
\lim_{\delta \downarrow 0} E \max_{0 \leq t \leq T}|q^{\mu,\delta}_t-q^{\mu}_t|^2 = 0
\end{equation*}
thanks to \cite[Theorem 7.2 Chapter VI]{IW01} and the following property stemming from the continuous differentiability of $q^{\mu}_t$ in $t$:

\begin{equation*}
\int_0^T \sigma(q^{\mu}_t) \, dw_t= \int_0^T \sigma(q^{\mu}_t) \circ dw_t,
\end{equation*}
where the integral on the right is understood in the Stratonovich sense.

What will happen if we tend $\mu \downarrow 0$  faster than $\delta \downarrow 0$? It turns out that with the smooth noise $w^{\delta}_t$, we have enough regularity to find the limit. This follows from the fact that for regular enough $f$,
\begin{equation*}
\lim_{\mu \downarrow 0}\int_0^t \cos \left( \frac{1}{\mu}\int_0^s \alpha(q^{\mu,\delta}_r ) \, dr \right)
f(q^{\mu,\delta}_s)ds=0.
\end{equation*}

We will discuss on this property in Lemma~\ref{th03}. Now, we are ready to state the main theorems.

\begin{thm}\label{th05} Under Hypothesis~\ref{hyp01}, there exists a constant $C>0$ such that for any $0 <\mu\le 1$ and $0<\delta\leq1$,
\begin{equation}
E \max_{0 \leq t \leq T}|q^{\mu,\delta}_t-q^{\delta}_t| \leq C
(1+T)^{\frac{11}{2}}\exp \left(\frac{C}{\delta^2}(1+T)^3\right)\mu,
\end{equation} 
where $q^{\delta}_t$ is the solution of the first order
differential equation
\begin{align}\label{eq19}
\left\{ \begin{array}{ll} \dot{q}^{\delta}_t=-A^{-1}(q^{\delta}_t) \,  b (q^{\delta}_t)+A^{-1}(q^{\delta}_t)\, \sigma(q^{\delta}_t) \, \dot{w}^{\delta}_t
\\\\
q^{\delta}_0=q_0 \in \mathbb{R}^2.
\end{array} \right.
\end{align}

In particular, for any fixed $0<\delta \leq 1$, 
\begin{equation*}
\lim_{\mu \downarrow 0}E\max_{0 \leq t \leq
T}|q^{\mu,\delta}_t-q^{\delta}_t|=0.
\end{equation*}

\end{thm}

We postpone the proof of Theorem~\ref{th05} to the end of this section. By \cite[Theorem 7.2 Chapter VI]{IW01}, we have the following result.

\begin{thm}\label{th06} Under Hypothesis~\ref{hyp01},
\begin{equation*}
\lim_{\delta \downarrow 0} E \max_{0 \leq t \leq T}
|q^{\delta}_t-\hat{q}_t|=0,
\end{equation*}
where $\hat{q}_t$ is the solution of the first order stochastic differential equation
\begin{equation}\label{neq04}
\left\{ \begin{array}{ll}
\dot{\hat{q}}_t=-A^{-1}(\hat{q}_t)b(\hat{q}_t)-(A^{-1}(\hat{q}_t)\sigma(\hat{q}_t))
 \circ \dot{w}_t
\\\\
\hat{q}_0=q_0  \in \mathbb{R}^2.
\end{array} \right.
\end{equation}

\qed

\end{thm}

We state the combination of the above two theorems in the following corollary.

\begin{cor}\label{th07} Under Hypothesis~\ref{hyp01}, $q^{\mu,\delta}_t$ converges to $\hat{q}_t$ in probability in $C([0,T];\RR^2)$ as $\mu \downarrow 0$ and $\delta \downarrow 0$ so that $\mu e^{\frac{C}{\delta^2}} \downarrow 0$ for each constant $C>0$.

\qed

\end{cor}

For the proof of Theorem~\ref{th05}, it is necessary to find some auxiliary bounds. In the following three lemmas, we find those bounds. 

First of all, in Lemma~\ref{th01}, we find a uniform bound of $|p^{\mu,\delta}_t|$ in $C([0,T];\RR^2)$ independent of $\mu$.

\begin{lem}\label{th01} Under Hypothesis~\ref{hyp01}, there exists a constant $C>0$ such that  for any $0 < \delta \leq 1$,
\begin{equation}
\sup_{\mu > 0}\max_{0\leq t \leq T}|p^{\mu,\delta}_t|\leq C(1+T)(1+X_T)\exp \left(\frac{C}{\delta}(1+T)(1+X_T) \right), \quad \mathbb{P}-a.s.,
\end{equation}
where
\begin{equation}\label{eq16}
X_T=\max_{0 \le t \le T+1}|w_t|.
\end{equation}

\end{lem}

\begin{proof}

Suppose $0\leq t \leq T$. From equation \eqref{eq2},
\[
\dot{p}^{\mu,\delta}_t-\frac{1}{\mu}\,A(q^{\mu,\delta}_t)
\, p^{\mu,\delta}_t =\frac{1}{\mu}\, b
(q^{\mu,\delta}_t)+\frac{1}{\mu}\,\sigma(q^{\mu,\delta}_t)
\,
\dot{w}^{\delta}_t.
\]

Multiplying both sides by
\[
\exp \left(-\frac{1}{\mu}\int_0^t A(q^{\mu,\delta}_s) \, ds \right),
\]
we get
\begin{align*}
&\left( \exp \left(-\frac{1}{\mu}\int_0^t
\,A(q^{\mu,\delta}_s) \, ds \right)\,
p^{\mu,\delta}_t \right)'\\ 
&\quad\qquad\qquad\qquad=\frac{1}{\mu}\,
\exp \left(-\frac{1}{\mu}\int_0^t \,A(q^{\mu,\delta}_s)
\, ds\right)\, b
(q^{\mu,\delta}_t)+\frac{1}{\mu}\,\exp \left(-\frac{1}{\mu}\int_0^t
\,A(q^{\mu,\delta}_s) \, ds \right)\,
\sigma(q^{\mu,\delta}_t) \, \dot{w}^{\delta}_t.
\end{align*}

Define $\beta^{\mu,\delta}_t$ as 
\[
\beta^{\mu,\delta}_t := \int_0^t \alpha(q^{\mu,\delta}_s) \,
ds.
\]

Considering the definition of $A(q^{\mu,\delta}_t)$ in \eqref{neq01}, we  have
\[\int_0^t\,A(q^{\mu,\delta}_s)
\,ds= \int_0^t \alpha(q^{\mu,\delta}_s)\,
ds \,A_0 =\beta^{\mu,\delta}_t A_0.
\] 

So, we may rewrite the above equation as 
\[
\left(\exp \left(-\frac{\beta^{\mu,\delta}_t}{\mu} A_0\right)\,
p^{\mu,\delta}_t \right)' =\frac{1}{\mu}\,
\exp \left(-\frac{\beta^{\mu,\delta}_t}{\mu} A_0 \right)\, b
(q^{\mu,\delta}_t)+\frac{1}{\mu}\exp \left(-\frac{\beta^{\mu,\delta}_t}{\mu} A_0 \right)\,
\sigma(q^{\mu,\delta}_t) \, \dot{w}^{\delta}_t.
\]

Integrating both sides with respect to $t$ we get
\begin{align}\notag
p^{\mu,\delta}_t &=\exp \left(\frac{\beta^{\mu,\delta}_t}{\mu} A_0
\right)p_0+\frac{1}{\mu}\,\exp \left(\frac{\beta^{\mu,\delta}_t}{\mu} A_0 \right) \int_0^t
\exp \left(-\frac{\beta^{\mu,\delta}_s}{\mu}A_0 \right) b (q^{\mu,\delta}_s)\, ds\\
&\notag\qquad\qquad\qquad\qquad\qquad\qquad+\frac{1}{\mu}\,\exp \left(\frac{\beta^{\mu,\delta}_t}{\mu} A_0 \right)\int_0^t
 \exp \left(-\frac{\beta^{\mu,\delta}_s}{\mu}A_0\right)\, \sigma(q^{\mu,\delta}_s)
\, dw^{\delta}_s\\
&\label{eq10}=:I_1(t)+I_2(t)+I_3(t).
\end{align}

By the definition of $A_0$ in \eqref{neq01}, we can calculate the matrix exponentials
\begin{equation}\label{eq12}
\exp \left(\pm\frac{\beta^{\mu,\delta}_t}{\mu}A_0\right)=\left(\begin{array}{ccc}
\cos \left(\frac{\beta^{\mu,\delta}_t}{\mu}\right) & \mp\sin \left(\frac{\beta^{\mu,\delta}_t}{\mu}\right) \\
\pm\sin \left(\frac{\beta^{\mu,\delta}_t}{\mu}\right) & \cos
\left(\frac{\beta^{\mu,\delta}_t}{\mu}\right)
\end{array} \right).
\end{equation}

Since \eqref{eq12} is an orthogonal matrix, for any $v \in
\mathbb{R}^2$,
\begin{equation}\label{neq05}
\left|\exp \left(\pm\frac{\beta^{\mu,\delta}_t}{\mu}A_0\right)v\right|=|v|
\end{equation}
so that 
\begin{equation}\label{neq06}
|I_1(t)|\leq |p_0|.
\end{equation}

As $A_0$ and $A_0^{-1}$ commute, we have
\begin{align}
\notag I_2(t)&=\exp \left(\frac{\beta^{\mu,\delta}_t}{\mu} A_0\right) \int_0^t
\left(-\frac{\alpha(q^{\mu,\delta}_s)}{\mu}A_0
\exp \left(-\frac{\beta^{\mu,\delta}_s}{\mu}A_0\right)\right)\left(-\frac{1}{\alpha(q^{\mu,\delta}_s)}A^{-1}_0
b (q^{\mu,\delta}_s)\right)ds\\
&\notag =\exp \left(\frac{\beta^{\mu,\delta}_t}{\mu} A_0\right)
\left(\left[\exp \left(-\frac{\beta^{\mu,\delta}_s}{\mu}A_0\right)\left(-\frac{1}{\alpha(q^{\mu,\delta}_s)}A^{-1}_0
b (q^{\mu,\delta}_s)\right)\right]^t_0\right.\\
&\notag \qquad \quad \left.-\int_0^t
\exp \left(\frac{\beta^{\mu,\delta}_s}{\mu} A_0\right) \left(\frac{\nabla
\alpha(q^{\mu,\delta}_s) \cdot
p^{\mu,\delta}_s}{\alpha(q^{\mu,\delta}_s)^2}
A^{-1}_0b(q^{\mu,\delta}_s)-\frac{1}{\alpha(q^{\mu,\delta}_s)}A^{-1}_0
D b (q^{\mu,\delta}_s)
p^{\mu,\delta}_s\right) ds\right)\\
&\notag =\exp \left(\frac{\beta^{\mu,\delta}_t}{\mu} A_0\right)
\left(-\frac{1}{\alpha(q^{\mu,\delta}_t)}A^{-1}_0
\exp \left(-\frac{\beta^{\mu,\delta}_t}{\mu} A_0\right) b
(q^{\mu,\delta}_t)+\frac{1}{\alpha(q_0)}A^{-1}_0 b(q_0) \right.\\
& \notag \qquad \quad \left. -\int_0^t
\exp \left(-\frac{\beta^{\mu,\delta}_s}{\mu} A_0\right) \left(\frac{\nabla
\alpha(q^{\mu,\delta}_s) \cdot
p^{\mu,\delta}_s}{\alpha(q^{\mu,\delta}_s)^2}
A^{-1}_0b(q^{\mu,\delta}_s)-\frac{1}{\alpha(q^{\mu,\delta}_s)}A^{-1}_0
D b (q^{\mu,\delta}_s) p^{\mu,\delta}_s\right) ds\right)\\
&\notag =-\frac{1}{\alpha(q^{\mu,\delta}_t)}A^{-1}_0 b
(q^{\mu,\delta}_t)+\frac{1}{\alpha(q_0)}A^{-1}_0\exp \left(\frac{\beta^{\mu,\delta}_t}{\mu} A_0\right) b(q_0) \\
&\label{neq07} \qquad -\int_0^t
\exp \left(\frac{\beta^{\mu,\delta}_t-\beta^{\mu,\delta}_s}{\mu}A_0 \right)\left(\frac{\nabla \alpha(q^{\mu,\delta}_s) \cdot
p^{\mu,\delta}_s}{\alpha(q^{\mu,\delta}_s)^2}
A^{-1}_0b(q^{\mu,\delta}_s)-\frac{1}{\alpha(q^{\mu,\delta}_s)}A^{-1}_0
D b (q^{\mu,\delta}_s) p^{\mu,\delta}_s\right)ds.
\end{align}

The same method can be used for $ I_3(t)$ and we get
\begin{align}
\notag I_3(t)&=\frac{1}{\mu}\exp \left(\frac{\beta^{\mu,\delta}_t}{\mu} A_0\right)\int_0^t
 \exp \left(-\frac{\beta^{\mu,\delta}_s}{\mu} A_0\right) \sigma(q^{\mu,\delta}_s)
 \dot{w}^{\delta}_s ds\\
&\notag =-\frac{1}{\alpha(q^{\mu,\delta}_t)}A^{-1}_0 \sigma
(q^{\mu,\delta}_t)\dot{w}^{\delta}_t+\frac{1}{\alpha(q_0)}A^{-1}_0\exp \left(\frac{\beta^{\mu,\delta}_t}{\mu} A_0\right) \sigma(q_0)\dot{w}^{\delta}_0 \\
&\notag\qquad -\int_0^t
\exp \left(\frac{\beta^{\mu,\delta}_t-\beta^{\mu,\delta}_s}{\mu}A_0\right)
\left(\frac{\nabla \alpha(q^{\mu,\delta}_s) \cdot
p^{\mu,\delta}_s}{\alpha(q^{\mu,\delta}_s)^2} A^{-1}_0
\sigma(q^{\mu,\delta}_s)\dot{w}^{\delta}_s \right.\\
&\label{neq08} \qquad\qquad\qquad\qquad\qquad\qquad\qquad \left.-\frac{1}{\alpha(q^{\mu,\delta}_s)}A^{-1}_0
D \sigma (q^{\mu,\delta}_s)
p^{\mu,\delta}_s\dot{w}^{\delta}_s -\frac{1}{\alpha(q^{\mu,\delta}_s)}A^{-1}_0
\sigma (q^{\mu,\delta}_s)\ddot{w}^{\delta}_s \right) ds.
\end{align}

To find bounds for $I_2(t)$ and $I_3(t)$, we need bounds for $\dot{w}^{\delta}_t$ and $\ddot{w}^{\delta}_t$. In view of equation \eqref{eq4}, we note that $(w^{\delta}_t)^{(n)}$, the $n$th derivative of $w^{\delta}_t$ with respect to $t$, satisfies
\begin{align*}
(w^{\delta}_t)^{(n)}&=\frac{(-1)^n}{\delta^n}\int_0^{\infty} w(s)\rho^{(n)} (\frac{s-t}{\delta})ds=\frac{(-1)^n}{\delta^n}\int_0^1 w(t+\delta r ) \rho^{(n)} (r)dr.
\end{align*}

Hence, for any $0 \leq t \leq T$,
\begin{align*}
|(w^{\delta}_t)^{(n)}| &\leq \frac{1}{\delta^n}\int_0^1|w(t+\delta s)|
|\rho^{(n)}(s)| ds\\
&\leq \frac{1}{\delta^n}\max_{0\leq t \leq
T+\delta}|w(t)|\int_0^1 |\rho^{(n)} (s)| ds=\frac{C(n)}{\delta^n}\max_{0\leq t \leq T+1}|w(t)|,
\end{align*}
where $C(n)$ is a constant depending on $n$.

Letting
\[
X_T:=\max_{0\leq t \leq T+1}|w_t|,
\]
we have
\[
\max_{0 \leq t \leq T}|(w^{\delta}_t)^{(n)}| \leq \frac{C(n)}{\delta^n}X_T.
\]

In particular, we can find a constant $C>0$ such that 
\begin{equation}\label{eq29}
\max_{0 \leq t \leq T}|\dot{w}^{\delta}_t| \leq \frac{C}{\delta}X_T
\end{equation}
and
\begin{equation*}
 \max_{0 \leq t \leq T}|\ddot{w}^{\delta}_t|
\leq \frac{C}{\delta^2}X_T.
\end{equation*}

Now, we are ready to find bounds for $I_2(t)$ and $I_3(t)$. Applying Hypothesis~\ref{hyp01}, \eqref{neq05}, and \eqref{eq29} to \eqref{neq07} and  \eqref{neq08}, we get
\begin{align}
\notag |I_2(t)|& \leq \left|\frac{1}{\alpha(q^{\mu,\delta}_t)}A^{-1}_0 b
(q^{\mu,\delta}_t) \right|+\left|\frac{1}{\alpha(q_0)}A^{-1}_0\exp \left(\frac{\beta^{\mu,\delta}_t}{\mu}A_0\right) b(q_0) \right| \\
&\notag+\int_0^t
\left|\exp \left(\frac{\beta^{\mu,\delta}_t-\beta^{\mu,\delta}_s}{\mu}A_0\right)
\left(\frac{\nabla \alpha(q^{\mu,\delta}_s) \cdot
p^{\mu,\delta}_s}{\alpha(q^{\mu,\delta}_s)^2}
A^{-1}_0b(q^{\mu,\delta}_s)-\frac{1}{\alpha(q^{\mu,\delta}_s)}A^{-1}_0
D b (q^{\mu,\delta}_s) p^{\mu,\delta}_s\right)\right| ds\\
&\label{neq09}\leq C_1+C_2+ C_3 \int^t_0|p^{\mu,\delta}_s|ds+ C_4 \int^t_0|p^{\mu,\delta}_s|ds=C_5+C_6 \int^t_0|p^{\mu,\delta}_s |ds
\end{align}
and 
\begin{align}
\notag |I_3(t)|&\leq \left|\frac{1}{\alpha(q^{\mu,\delta}_t)}A^{-1}_0 \sigma
(q^{\mu,\delta}_t)\dot{w}^{\delta}_t \right|+\left|\frac{1}{\alpha(q)}A^{-1}_0\exp \left(\frac{\beta^{\mu,\delta}_t}{\mu}A_0\right)\sigma(q)\dot{w}^{\delta}_0 \right| \\
&\notag\quad +\int_0^t
\left|\exp \left(\frac{\beta^{\mu,\delta}_t-\beta^{\mu,\delta}_s}{\mu}A_0\right)
\left(\frac{\nabla \alpha(q^{\mu,\delta}_s) \cdot
p^{\mu,\delta}_s}{\alpha(q^{\mu,\delta}_s)^2} A^{-1}_0
\sigma(q^{\mu,\delta}_s)\dot{w}^{\delta}_s-\frac{1}{\alpha(q^{\mu,\delta}_s)}A^{-1}_0
D \sigma (q^{\mu,\delta}_s)
p^{\mu,\delta}_s\dot{w}^{\delta}_s \right.\right.\\
&\notag\qquad\qquad\qquad\qquad\qquad\qquad\qquad\qquad\qquad\qquad\qquad\qquad\qquad\quad\,\,\,\, \left.\left.+\frac{1}{\alpha(q^{\mu,\delta}_s)}A^{-1}_0
\sigma (q^{\mu,\delta}_s)\ddot{w}^{\delta}_s\right)\right| ds\\
&\notag\leq \frac{C_7}{\delta}X_T+\frac{C_8}{\delta}X_T+\frac{C_9}{\delta}X_T
\int^t_0|p^{\mu,\delta}_s|ds+\frac{C_{10}}{\delta}X_T
\int^t_0|p^{\mu,\delta}_s|ds+\frac{C_{11} t}{\delta^2}X_T\\
&\label{neq10}\leq \frac{C_{12}}{\delta^2}(1+t)X_T+\frac{C_{13}}{\delta}X_T
\int^t_0|p^{\mu,\delta}_s|ds.
\end{align}

Applying three inequalities \eqref{neq06},  \eqref{neq09}, and \eqref{neq10} to \eqref{eq10}, we get a bound for $p^{\mu,\delta}_t$.
\begin{align*}
|p^{\mu,\delta}_t|&\leq |I_1(t)|+|I_2(t)|+|I_3(t)|\\
&\leq |p|+C_5+C_6 \int^t_0|p^{\mu,\delta}_s|ds+\frac{C_{12}}{\delta^2}(1+t)X_T+\frac{C_{13}}{\delta}X_T
\int^t_0|p^{\mu,\delta}_s|ds\\
&\leq \frac{C_{14}}{\delta^2}(1+t)(1+X_T)+\frac{C_{15}}{\delta}(1+X_T)
\int^t_0|p^{\mu,\delta}_s|ds.
\end{align*}

By Gronwall's lemma,
\begin{align*}
|p^{\mu,\delta}_t|&\leq  \frac{C_{14}}{\delta^2}(1+t)(1+X_T)\exp \left(\frac{C_{15}}{\delta}(1+X_T)t\right)\\
&\leq C(1+t)(1+X_T)\exp \left(\frac{C}{\delta}(1+t)(1+X_T) \right)
\end{align*}
for sufficiently large $C>0$. The last inequality came from the fact that the term $\frac{1}{\delta^2}$ can be absorbed in the term $e^{\frac{C}{\delta}}$ for large $C$.

So, we have
\begin{equation*}
\max_{0\leq t \leq T}|p^{\mu,\delta}_t|\leq
C(1+T)(1+X_T)\exp \left(\frac{C}{\delta}(1+T)(1+X_T)\right).
\end{equation*}

\end{proof}

\begin{rem}\label{th02}
Note that by Lemma~\ref{th01},  $q^{\mu,\delta}_t$ is Lipschitz continuous with its Lipschitz constant independent of $\mu$ on the interval $[0,T]$. That is, for $0\leq t_1
\le t_2 \le T$,

\begin{equation}
|q^{\mu,\delta}_{t_2}-q^{\mu,\delta}_{t_1}|\leq
C(T,\delta,X_T)|t_2-t_1| \quad \mathbb{P}-a.s.
\end{equation}

\end{rem}

Next, we find a bound of the integral of a highly oscillating function. The result is similar to that of the Riemann-Lebesgue lemma.

\begin{lem}\label{th03} Under Hypothesis~\ref{hyp01}, there exists a constant $C>0$ such that for any $\mu>0$ and $0<\delta \le 1$, and for any bounded Lipschitz
continuous function $f : \RR^2 \rightarrow \RR$ with the Lipschitz constant $K_f$,
\begin{equation}
\label{eq27}\left|\int_0^t \cos \left(\frac{\beta^{\mu,\delta}_s}{\mu}\right)
f(q^{\mu,\delta}_s)ds \right|+\left|\int_0^t \sin \left(\frac{\beta^{\mu,\delta}_s}{\mu}\right)
f(q^{\mu,\delta}_s)ds\right| \leq C_1(t,\delta,X_t, f)\mu
\end{equation}
and
\begin{equation}
\label{eq28}\left|\int_0^t \cos \left(\frac{\beta^{\mu,\delta}_s}{\mu}\right)
f(q^{\mu,\delta}_s)\dot{w}^{\delta}_s ds\right|+\left|\int_0^t \sin
\left(\frac{\beta^{\mu,\delta}_s}{\mu}\right)
f(q^{\mu,\delta}_s)\dot{w}^{\delta}_s ds\right| \leq C_2(t,\delta,X_t,
f)\mu\\ 
\end{equation}
$\mathbb{P}-a.s.$, where
\begin{equation*} 
C_1(t,\delta,X_t,
f)=C(1+t)^2(|f|_{\infty}+K_f)(1+X_t)\exp \left( \frac{C}{\delta}(1+t)(1+X_t)\right)
\end{equation*}
and 
\begin{equation*}
C_2(t,\delta,X_t, f)=C
(1+t)^2(|f|_{\infty}+K_f)(1+X_t)^2\exp \left( \frac{C}{\delta}(1+t)(1+X_t)\right).
\end{equation*}

\end{lem}

\begin{proof}

Since $\alpha(q^{\mu,\delta}_s)$ is strictly positive,  $\beta^{\mu,\delta}_t$ is strictly increasing, so that
\begin{equation*}
u=\frac{\beta^{\mu,\delta}_s}{\mu}
\end{equation*}
provides a good change of variables.

Then, as
\begin{equation}\label{neq13}
du=\frac{\alpha(q^{\mu,\delta}_s)}{\mu}ds,
\end{equation}
we have
\begin{equation*}
\int_0^t \cos \left(\frac{\beta^{\mu,\delta}_s}{\mu}\right)\,f(q^{\mu,\delta}_s)ds=\mu\int_0^{\frac{\beta^{\mu,\delta}_t}{\mu}}
\cos (u)
\,\,\frac{f(q^{\mu,\delta}_{s(u)})}{\alpha(q^{\mu,\delta}_{s(u)})}du.
\end{equation*}

If we define
\begin{equation}\label{neq11}
 g^{\mu,\delta}(u):=\frac{f(q^{\mu,\delta}_{s(u)})}{\alpha(q^{\mu,\delta}_{s(u)})},
\end{equation}
we get
\begin{align}
\notag &\left| \int_0^t\cos \left(\frac{\beta^{\mu,\delta}_s}{\mu}\right)
\,f(q^{\mu,\delta}_s)ds \right|=\left|\mu\int_0^{\frac{\beta^{\mu,\delta}_t}{\mu}}
\cos (u)
\,g^{\mu,\delta}(u)du\right|\\
&\notag\qquad\qquad =\left|\mu \sum^{\left \lfloor \frac{\beta^{\mu,\delta}_t}{2\pi\mu}
\right \rfloor-1}_{k=0} \int^{2\pi (k+1)}_{2\pi k} \cos (u)
\,g^{\mu,\delta}(u)du+\mu\int^{\frac{\beta^{\mu,\delta}_t}{\mu}}_{2 \pi
\left \lfloor
\frac{\beta^{\mu,\delta}_t}{2\pi\mu} \right \rfloor}\cos (u)g^{\mu,\delta}(u) du \right|\\
&\notag\qquad\qquad \le\left|\mu \sum^{\left \lfloor \frac{\beta^{\mu,\delta}_t}{2\pi\mu}
\right \rfloor-1}_{k=0} \int^{2\pi (k+1)}_{2\pi k} \cos (u)
\,\left(\left(g^{\mu,\delta}(u)-g^{\mu,\delta}(2\pi k )\right)+g^{\mu,\delta}(2\pi k
)\right)du \right|\\
&\notag\qquad\qquad\qquad\qquad\qquad\qquad\qquad\qquad\qquad\qquad\qquad\qquad\qquad\qquad\quad+\left|\mu \int^{\frac{\beta^{\mu,\delta}_t}{\mu}}_{2 \pi \left \lfloor
\frac{\beta^{\mu,\delta}_t}{2\pi\mu} \right \rfloor}\cos (u)g^{\mu,\delta}(u) du\right|\\
&\notag\qquad\qquad =\left|\mu \sum^{\left \lfloor \frac{\beta^{\mu,\delta}_t}{2\pi\mu}
\right \rfloor-1}_{k=0} \int^{2\pi (k+1)}_{2\pi k} \cos (u)
\,\left(g^{\mu,\delta}(u)-g^{\mu,\delta}(2\pi k )\right)du\right|+\left|\mu
\int^{\frac{\beta^{\mu,\delta}_t}{\mu}}_{2 \pi \left \lfloor
\frac{\beta^{\mu,\delta}_t}{2\pi\mu} \right \rfloor} \cos (u)g^{\mu,\delta}(u) du\right|\\
&\label{eq15} \qquad\qquad  = |I_1(t)|+|I_2(t)|.
\end{align}

We first find a bound of $|I_1(t)|$. Note that from Hypothesis~\ref{hyp01}, 
\begin{equation*}
0<\alpha_0\leq \alpha(q) \leq |\alpha|_{\infty}.
\end{equation*}

So,
\begin{align}
&\notag\left|g^{\mu,\delta}(u)-g^{\mu,\delta}(2\pi k
)\right|=\left|\frac{f(q^{\mu,\delta}_{s(u)})}{\alpha(q^{\mu,\delta}_{s(u)})}-\frac{f(q^{\mu,\delta}_{s(2\pi
k)})}{\alpha(q^{\mu,\delta}_{s(2\pi
k)})}\right|\\
&\notag\qquad\qquad\qquad=\frac{1}{\alpha(q^{\mu,\delta}_{s(u)})\alpha(q^{\mu,\delta}_{s(2\pi
k)})}\left|f(q^{\mu,\delta}_{s(u)})\alpha(q^{\mu,\delta}_{s(2\pi
k)})-f(q^{\mu,\delta}_{s(2\pi k)})\alpha(q^{\mu,\delta}_{s(u)})\right|\\
&\notag\qquad\qquad\qquad\le
\frac{1}{\alpha_0^2}\left|f(q^{\mu,\delta}_{s(u)})\alpha(q^{\mu,\delta}_{s(2\pi
k)})-f(q^{\mu,\delta}_{s(2\pi k)})\alpha(q^{\mu,\delta}_{s(u)})\right|\\
&\notag\qquad\qquad\qquad\le
\frac{1}{\alpha_0^2}\left(\left|f(q^{\mu,\delta}_{s(u)})\alpha(q^{\mu,\delta}_{s(2\pi
k)})-f(q^{\mu,\delta}_{s(u)})\alpha(q^{\mu,\delta}_{s(u)})\right|+\left|f(q^{\mu,\delta}_{s(u)})\alpha(q^{\mu,\delta}_{s(u)})-f(q^{\mu,\delta}_{s(2\pi
k)})\alpha(q^{\mu,\delta}_{s(u)})\right|\right)\\
&\notag\qquad\qquad\qquad \leq
\frac{1}{\alpha_0^2}\left(|f|_{\infty}K\left|q^{\mu,\delta}_{s(u)}-q^{\mu,\delta}_{s(2\pi
k)}\right|+|\alpha|_{\infty} K_f\left|q^{\mu,\delta}_{s(u)}-q^{\mu,\delta}_{s(2\pi
k)}\right|\right)\\
&\label{neq12}\qquad\qquad\qquad \leq C_1(|f|_{\infty}+K_f)\left|q^{\mu,\delta}_{s(u)}-q^{\mu,\delta}_{s(2\pi k)}\right|,
\end{align}

where $K$ is the Lipschitz constant for $\alpha(q)$ and $K_f$ is
the Lipschitz constant for $f(q)$.\\

From Lemma~\ref{th01}, we have
\[
\label{eq30}\left|\frac{d}{du}q^{\mu,\delta}_{s(u)}\right|=\left|p^{\mu,\delta}_{s(u)} \frac{ds(u)}{du}\right|=\left|p^{\mu,\delta}_{s(u)} \frac{\mu}{\alpha(q^{\mu,\delta}_{s(u)})}\right| \leq C(t,\delta,X_t) \frac{\mu}{\alpha_0},
\]
where 
\begin{equation}\label{nnneq01}
C(t,\delta,X_t):=C(1+t)(1+X_t)\exp \left(\frac{C}{\delta}(1+t)(1+X_t) \right).
\end{equation}

So, from \eqref{neq12},
\begin{equation*}
|g^{\mu,\delta}(u)-g^{\mu,\delta}(2\pi k )| \leq
C_2C(t,\delta,X_t)(|f|_{\infty}+K_f)\mu |u - 2\pi k|.
\end{equation*}

This implies
\begin{align}
\notag |I_1(t)| &\leq \mu \sum^{\left \lfloor
\frac{\beta^{\mu,\delta}_t}{2\pi\mu} \right \rfloor-1}_{k=0} \int^{2\pi
(k+1)}_{2\pi k} |\cos (u)| |g^{\mu,\delta}(u)-g^{\mu,\delta}(2\pi k
)|du\\
&\notag \leq \mu \sum^{\left \lfloor \frac{\beta^{\mu,\delta}_t}{2\pi\mu}
\right \rfloor-1}_{k=0} \int^{2\pi (k+1)}_{2\pi k} C_2C(t,\delta,X_t)(|f|_{\infty}+K_f)\mu (u - 2\pi k) du\\
&\notag = \mu \sum^{\left \lfloor \frac{\beta^{\mu,\delta}_t}{2\pi\mu}
\right \rfloor-1}_{k=0} C_2C(t,\delta,X_t)(|f|_{\infty}+K_f)\mu 2 \pi^2\\
&\notag=
C_2C(t,\delta,X_t)(|f|_{\infty}+K_f) \mu^2 2 \pi^2 \left \lfloor
\frac{\beta^{\mu,\delta}_t}{2\pi\mu}
\right \rfloor.
\end{align}

Since 
\begin{equation*}
\beta^{\mu,\delta}_t \leq |\alpha|_{\infty} t,
\end{equation*}
we get
\begin{align}
\notag |I_1(t)| &\leq C_2C(t,\delta,X_t)(|f|_{\infty}+K_f) \mu  \pi
|\alpha|_{\infty} t \\
&\label{eq13}=C_3C(t,\delta,X_t)(|f|_{\infty}+K_f)t\mu.
\end{align}

A bound for $|I_2(t)|$ can be found relatively easily. From \eqref{neq11},
\begin{align}
\notag |I_2(t)| &\leq \mu \int^{\frac{\beta^{\mu,\delta}_t}{\mu}}_{2
\pi \left \lfloor \frac{\beta^{\mu,\delta}_t}{2\pi\mu}
\right \rfloor}\left|\frac{f(q^{\mu,\delta}_{s(u)})}{\alpha(q^{\mu,\delta}_{s(u)})}\right| du\\
&\notag \leq \mu
\int^{\frac{\beta^{\mu,\delta}_t}{\mu}}_{2 \pi \left \lfloor
\frac{\beta^{\mu,\delta}_t}{2\pi\mu}
\right \rfloor}\frac{|f|_{\infty}}{\alpha_0} du\\
&\label{eq14}\leq \mu 2 \pi
\frac{|f|_{\infty}}{\alpha_0}=C_4|f|_{\infty}\mu.
\end{align}

So, from \eqref{eq15},\eqref{eq13}, \eqref{eq14}, and \eqref{nnneq01}, 
\begin{align*}
\notag \left|\int_0^t \cos \left(\frac{\beta^{\mu,\delta}_s}{\mu}\right)
\,f(q^{\mu,\delta}_s)ds\right|&\leq
C_3C(t,\delta,X_t)(|f|_{\infty}+K_f)t\mu+C_4|f|_{\infty}\mu\\
&\leq
C_5(1+t)^2(|f|_{\infty}+K_f)(1+X_t)\exp \left( \frac{C_6}{\delta}(1+t)(1+X_t)\right)\mu.
\end{align*}

This proves inequality \eqref{eq27} for the cosine part. The sine part can be treated analogously. Now consider inequality \eqref{eq28}. As in \eqref{eq15},
\begin{align}
\notag \left|\int_0^t \cos \left(\frac{\beta^{\mu,\delta}_s}{\mu}\right)
f(q^{\mu,\delta}_s)\dot{w}^{\delta}_s ds \right|&\leq \left|\mu \sum^{\left \lfloor
\frac{\beta^{\mu,\delta}_t}{2\pi\mu} \right \rfloor-1}_{k=0} \int^{2\pi
(k+1)}_{2\pi k} \cos (u) \,\left(g^{\mu,\delta}_1(u)-g^{\mu,\delta}_1(2\pi
k )\right)du\right|\\
&\notag \qquad\qquad\qquad\qquad\qquad\qquad\qquad\qquad\qquad+\left|\mu \int^{\frac{\beta^{\mu,\delta}_t}{\mu}}_{2 \pi \left \lfloor
\frac{\beta^{\mu,\delta}_t}{2\pi\mu} \right \rfloor}g^{\mu,\delta}_1(u) du \right|\\
&\label{neq14}=|I_1(t)|+|I_2(t)|,
\end{align}
where
\begin{equation*}
g^{\mu,\delta}_1(u):=\frac{f(q^{\mu,\delta}_{s(u)})}{\alpha(q^{\mu,\delta}_{s(u)})}\dot{w}^{\delta}_{s(u)}.
\end{equation*}

By a similar argument as in \eqref{neq12}, we obtain
\begin{align*}
&\left|g^{\mu,\delta}_1(u)-g^{\mu,\delta}_1(2\pi k )\right|=\left|\frac{f(q^{\mu,\delta}_{s(u)})}{\alpha(q^{\mu,\delta}_{s(u)})}\dot{w}^{\delta}_{s(u)}-\frac{f(q^{\mu,\delta}_{s(2\pi k)})}{\alpha(q^{\mu,\delta}_{s(2\pi k)})}\dot{w}^{\delta}_{s(2\pi k)}\right|\\
& \qquad\qquad\qquad\leq
\frac{1}{\alpha_0^2}\left(|f|_{\infty}\max_{0 \leq s \leq
t}\left\{|\dot{w}^{\delta}_s|\right\}K\left|q^{\mu,\delta}_{s(u)}-q^{\mu,\delta}_{s(2\pi
k)}\right|+|\alpha|_{\infty} \max_{0 \leq s \leq
t}\left\{|\dot{w}^{\delta}_s|\right\}
K_f \left|q^{\mu,\delta}_{s(u)}-q^{\mu,\delta}_{s(2\pi
k)}\right|\right.\\
&\qquad\qquad\qquad\qquad\qquad\qquad\qquad\qquad\qquad\qquad\qquad\qquad+\left.|\alpha|_{\infty} |f|_{\infty}\max_{0 \leq s \leq
t}\left\{|\ddot{w}^{\delta}_s|\right\}|s(u)-s(2\pi k)|\right).
\end{align*}

Considering inequalities in \eqref{eq29} and Remark~\ref{th02},
\begin{align}
\notag &\left|g^{\mu,\delta}_1(u)-g^{\mu,\delta}_1(2\pi k )\right|\leq \frac{C_7}{\delta}(|f|_{\infty}+K_f)X_t\left|q^{\mu,\delta}_{s(u)}-q^{\mu,\delta}_{s(2\pi k)}\right|+\frac{C_8}{\delta^2}|f|_{\infty}X_t|s(u)-s(2\pi k)|\\
&\qquad\qquad\qquad\notag \le \frac{C_7}{\delta}(|f|_{\infty}+K_f)X_tC(t,\delta,X_t)|s(u)-s(2\pi k)|+\frac{C_8}{\delta^2}|f|_{\infty}X_t|s(u)-s(2\pi k)|\\
&\qquad\qquad\qquad\notag \leq \frac{C_9}{\delta^2}\left(\delta(|f|_{\infty}+K_f)C(t,\delta,X_t) + |f|_{\infty}\right)X_t|s(u)-s(2\pi k)|.
\end{align}

Note that from \eqref{neq13}
\begin{equation*}
|\frac{d}{du}s(u)|\leq\frac{\mu}{\alpha_0}.
\end{equation*}

Therefore, if we assume that $0 < \delta \leq 1$, we have
\begin{align}
\notag \left|g^{\mu,\delta}_1(u)-g^{\mu,\delta}_1(2\pi k )\right|&\leq \frac{C_9}{\delta^2}(|f|_{\infty}+K_f)C(t,\delta,X_t) +|f|_{\infty})X_t\frac{\mu}{\alpha_0}|u-2\pi k|\\
&\notag \leq  \frac{C_{10}}{\delta^2}(|f|_{\infty}+K_f)C(t,\delta,X_t)X_t\mu|u-2\pi k|.
\end{align}

By the same procedures as in \eqref{eq13} and \eqref{eq14},
\begin{equation*}
|I_1(t)| \leq  \frac{C_{10}}{\delta^2}(|f|_{\infty}+K_f)C(t,\delta,X_t)X_t t\mu
\end{equation*}
and
\begin{equation*}
|I_2(t)| \leq \frac{C_{11}}{\delta}|f|_{\infty}X_t\mu.
\end{equation*}

Now from \eqref{neq14}, we get
\begin{align*}
\notag \left|\int_0^t \cos \frac{\beta^{\mu,\delta}_s}{\mu}
\,f(q^{\mu,\delta}_s)\dot{w}^{\delta}_s ds\right|&\leq \frac{C_{10}}{\delta^2}(|f|_{\infty}+K_f)C(t,\delta,X_t)X_t t\mu+\frac{C_{11}}{\delta}|f|_{\infty}X_t\mu\\
&\leq C_{12}
(1+t)^2(|f|_{\infty}+K_f)(1+X_t)^2\exp \left(\frac{C_{13}}{\delta}(1+t)(1+X_t)\right)\mu.
\end{align*}

The last inequality was from \eqref{nnneq01} and the fact that $\frac{1}{\delta}$ or $\frac{1}{\delta^2}$ can be
absorved in $e^{\frac{c}{\delta}}$ in $C(t,\delta,X_t)$.

\end{proof}

In the next lemma, we show that the expectation of the exponential of the uniform norm of the two dimensional Wiener process in $C([0,t+1]; \RR^2)$ is finite. This property will be used at the end of the proof of the main theorem.

\begin{lem}\label{th04} For any integer $n \geq 0$ and $a \in \RR^+$, there exists a constant $C(n)>0$ such that 
\begin{equation}
E\left((1+X_t)^n e^{(1+X_t)a}\right) \leq
C(n)(1+t)^{\frac{n}{2}}e^{2(1+t)a^2+a},
\end{equation}
where
\begin{equation*}
X_t=\max_{0 \le s \le t+1}|w_s|.
\end{equation*}
\end{lem}

\begin{proof}

We have 
\begin{equation*}
E\left((1+X_t)^n e^{(1+X_t)a}\right)\leq C_1(n)e^{a} \sum_{k=0}^{n}E(X_t^{k}e^{a X_t}) \leq  C_1(n) e^{a}
\sum_{k=0}^{n}\left(E(X_t^{2k})\right)^{\frac{1}{2}}\left(E(e^{2a X_t})\right)^{\frac{1}{2}},
\end{equation*}
where $C_1(n)$ is a constant depending on $n$.

Since $w_s=(w^1_s, w^2_s)$, where $w^1_s$ and $w^2_s$ are independent one dimensional Wiener processes, defining
\begin{equation*}
 X_{i,t}:=\max_{0 \le s \le t+1}|w^i_s|
\end{equation*}
for $i=1, 2$, we have
\begin{equation*}
X_t\le X_{1,t}+X_{2,t}.
\end{equation*}

From
\begin{align*}
 E(X_t^{2k})&\leq E((X_{1,t}+X_{2,t})^{2k}) \leq C_2(n) E(X_{1,t}^{2k}+X_{2,t}^{2k})\\
&=2C_2(n)E(X_{1,t}^{2k})
\end{align*}
for $k=1,2,...,n$ and
\begin{equation*}
E(e^{2a X_t})\leq E(e^{2a( X_{1,t}+X_{2,t})})=E(e^{2a X_{1,t}})^2,
\end{equation*}
we get
\begin{equation}\label{eq18}
E\left((1+X_t)^n e^{(1+X_t)a}\right) \leq C_3(n)e^a\sum_{k=0}^{n}\left(E(X_{1,t}^{2k})\right)^{\frac{1}{2}}E(e^{2a X_{1,t}}).
\end{equation}

To find bounds for $E(X_{1,t}^{2k})$ or $E(e^{2a X_{1,t}})$, we need to know a bound for the distribution of $X_{1,t}$. We use the symmetry of the Wiener process and the reflection principle  to find this bound. 

For $x\geq 0$,
\begin{align}
\notag P(\max_{0\leq s \leq T}|w^1_s|>x)&= P(\{\max_{0\leq s \leq T}\{w^1_s\}>x\} \cup \{\min_{0\leq s \leq T} \{w^1_s\}<-x\})\\
&\notag \leq  P(\max_{0\leq s \leq T}\{w^1_s\}>x)+ P(\min_{0\leq s \leq T} \{w^1_s\}<-x)\\
&\notag =2P(\max_{0\leq s \leq T}\{w^1_s\}>x).
\end{align}

By the reflection principle, 
\[
P(\max_{0\leq s \leq T}|w^1_s|>x) \leq 4P(w^1_T > x)=4 \int_x^{\infty}\frac{1}{\sqrt{2\pi
T}}e^{-\frac{y^2}{2T}}dy.
\]

So, for $T=t+1$,
\begin{equation}\label{eq17}
P(X_t>x) \leq 4 \int_x^{\infty}\frac{1}{\sqrt{2\pi
(t+1)}}e^{-\frac{y^2}{2(t+1)}}dy.
\end{equation}

Using inequality \eqref{eq17},
\begin{align*}
E(e^{2a X_{1,t}})&=\int_0^{\infty} P(e^{2a X_{1,t}}>x)dx=\int_0^{\infty} P(X_{1,t}>\frac{1}{2a}\ln x)dx\\
&\leq \int_0^{\infty} 4 P(w^1_{t+1}>\frac{1}{2a}\ln x)dx=4
\int_0^{\infty}\int_{\frac{1}{2a}\ln
x}^{\infty}\frac{1}{\sqrt{2\pi
(t+1)}}e^{-\frac{y^2}{2(t+1)}}dy\,dx\\
&\leq 4e^{2(t+1)a^2}
\end{align*}

and
\begin{align*}
E(X_{1,t}^{2k}) &=\int_0^{\infty} P(X_{1,t}^{2k}>x)dx=\int_0^{\infty} P(X_{1,t}>x^{\frac{1}{2k}})dx\\
&\leq \int_0^{\infty} 4 P(w^1_{t+1}>x^{\frac{1}{2k}})dx=4\int_0^{\infty}\int_{x^{\frac{1}{2k}}}^{\infty}\frac{1}{\sqrt{2\pi
(t+1)}}e^{-\frac{y^2}{2(t+1)}}dy\,dx\\
&\leq C_4(n)(t+1)^k
\end{align*}
for $k=1,2,...,n$.

Applying these bounds to \eqref{eq18},
\begin{align*}
E\left((1+X_t)^n e^{(1+X_t)a}\right)&\leq C_5(n) e^a\sum_{k=0}^{n} (t+1)^{\frac{k}{2}}e^{2(t+1)a^2}\\
&\leq C_6(n) (t+1)^{\frac{n}{2}}e^{2(t+1)a^2+a}.
\end{align*}

The last inequality was from Young's inequality:
\begin{equation*}
\sum_{k=0}^{n} (t+1)^{\frac{k}{2}}\leq C_7(n)(1+(t+1)^{\frac{n}{2}}) \leq 2C_7(n)(t+1)^{\frac{n}{2}}.
\end{equation*}

\end{proof}

Finally, we are ready to prove the main theorem, Theorem~\ref{th05}.

\begin{proof}[Proof of Theorem~\ref{th05}]
	
Consider $0\leq t \leq T$. First, we find representations of $q^{\mu,\delta}_t$ and $q^{\delta}_t$. Integrating equations \eqref{eq10} and  \eqref{eq19},
\begin{align*}
\notag q^{\mu,\delta}_t=q_0+\int_0^t
\exp \left(\frac{\beta^{\mu,\delta}_s}{\mu} A_0 \right)p_0 \, ds
+&\frac{1}{\mu}\,\int_0^t
\exp \left(\frac{\beta^{\mu,\delta}_s}{\mu} A_0 \right)\int_0^s
\exp \left(-\frac{\beta^{\mu,\delta}_r}{\mu} A_0 \right) b (q^{\mu,\delta}_r)\, dr
\, ds\\
&+ \frac{1}{\mu}\,\int_0^t
\exp \left(\frac{\beta^{\mu,\delta}_s}{\mu} A_0 \right)\int_0^s
\exp \left(-\frac{\beta^{\mu,\delta}_r}{\mu} A_0 \right) \sigma(q^{\mu,\delta}_r)
\, dw^{\delta}_r \, ds
\end{align*}
and
\begin{align*}
\notag q^{\delta}_t&=q_0-\int_0^t A^{-1}(q^{\delta}_s) \,
b(q^{\delta}_s) \, ds - \int_0^t A^{-1}(q^{\delta}_s) \,
\sigma(q^{\delta}_s) \,
dw^{\delta}_s\\
&=q_0-\int_0^t \frac{1}{\alpha(q^{\delta}_s)}A^{-1}_0 \,
b(q^{\delta}_s) \, ds - \int_0^t
\frac{1}{\alpha(q^{\delta}_s)}A^{-1}_0 \, \sigma(q^{\delta}_s) \,
dw^{\delta}_s.
\end{align*}

Subtracting $q^{\delta}_t$ from $q^{\mu,\delta}_t$,
\begin{align}
\notag q^{\mu, \delta}_t-q^{\delta}_t&=\int_0^t \exp \left(\frac{\beta^{\mu,\delta}_s}{\mu} A_0 \right)p_0 \, ds\\
&\notag+\left(\frac{1}{\mu}\,\int_0^t
\exp \left(\frac{\beta^{\mu,\delta}_s}{\mu} A_0 \right)\int_0^s
\exp \left(-\frac{\beta^{\mu,\delta}_r}{\mu} A_0 \right) b (q^{\mu,\delta}_r)\, dr
\, ds+\int_0^t \frac{1}{\alpha(q^{\mu,\delta}_s)}A^{-1}_0 \,
b(q^{\mu,\delta}_s) \, ds\right)\\
&\notag+\left(\frac{1}{\mu}\,\int_0^t
\exp \left(\frac{\beta^{\mu,\delta}_s}{\mu} A_0 \right)\int_0^s
 \exp \left(-\frac{\beta^{\mu,\delta}_r}{\mu} A_0 \right) \sigma(q^{\mu,\delta}_r)
\, dw^{\delta}_r \, ds+\int_0^t
\frac{1}{\alpha(q^{\mu,\delta}_s)}A^{-1}_0 \,
\sigma(q^{\mu,\delta}_s) \, dw^{\delta}_s\right)\\
&\notag-\left(\int_0^t \frac{1}{\alpha(q^{\mu,\delta}_s)}A^{-1}_0
\, b(q^{\mu,\delta}_s) \, ds-\int_0^t
\frac{1}{\alpha(q^{\delta}_s)}A^{-1}_0 \,
b(q^{\delta}_s) \, ds\right)\\
&\notag-\left(\int_0^t \frac{1}{\alpha(q^{\mu,\delta}_s)}A^{-1}_0
\, \sigma(q^{\mu,\delta}_s) \, dw^{\delta}_s-\int_0^t
\frac{1}{\alpha(q^{\delta}_s)}A^{-1}_0 \,
\sigma(q^{\delta}_s) \, dw^{\delta}_s\right)\\
&\label{eq25}=I_1(t)+I_2(t)+I_3(t)+I_4(t)+I_5(t).
\end{align}

To get a bound for $q^{\mu, \delta}_t-q^{\delta}_t$, we will find bounds for the terms from $I_1(t)$ to $I_5(t)$. First, consider $I_1(t)$.

From \eqref{eq12}, expressing
$p_0=\left(\begin{array}{ccc} p^1_0 \\ p^2_0 \end{array} \right)$, we have

\begin{align*}
|I_1(t)|&=\left| \left(\begin{array}{ccc}
\int_0^t \cos (\frac{\beta^{\mu,\delta}_s}{\mu})ds \,p^1_0-\int_0^t \sin (\frac{\beta^{\mu,\delta}_s}{\mu})ds \,p^2_0 \\
\int_0^t \sin (\frac{\beta^{\mu,\delta}_s}{\mu})ds \,p^1_0+ \int_0^t
\cos(\frac{\beta^{\mu,\delta}_s}{\mu}) ds \,p^2_0
\end{array} \right)\right|\\
&\leq |p_0|\left(\left|\int_0^t \cos
(\frac{\beta^{\mu,\delta}_s}{\mu})ds\right|+\left|\int_0^t \sin
(\frac{\beta^{\mu,\delta}_s}{\mu})ds\right|\right)\\
&\leq C|p_0|(1+t)^2(1+X_t)e^{\frac{C}{\delta}(1+t)(1+X_t)}\mu.
\end{align*}

In the last inequality, we used Lemma~\ref{th03}.

Now, let's consider $I_2(t)$. Note that the commutativity of $A_0$ and
$A_0^{-1}$ justifies the commutativity of matrix exponentials. Applying integration by parts,
\begin{align}
\notag &\frac{1}{\mu}\,\int_0^t
\exp \left(\frac{\beta^{\mu,\delta}_s}{\mu} A_0 \right)\int_0^s
\exp \left(-\frac{\beta^{\mu,\delta}_r}{\mu} A_0 \right) b (q^{\mu,\delta}_r)\, dr
\, ds\\
&\notag=\int_0^t \frac{\alpha(q^{\mu,\delta}_s)}{\mu} A_0
\exp \left(\frac{\beta^{\mu,\delta}_s}{\mu} A_0 \right)
\frac{1}{\alpha(q^{\mu,\delta}_s)}A^{-1}_0 \int_0^s
\exp \left(-\frac{\beta^{\mu,\delta}_r}{\mu} A_0 \right) b (q^{\mu,\delta}_r)\, dr \, ds\\
\notag &=\left[\exp \left(\frac{\beta^{\mu,\delta}_s}{\mu} A_0 \right)
\frac{1}{\alpha(q^{\mu,\delta}_s)}A^{-1}_0 \int_0^s
\exp \left(-\frac{\beta^{\mu,\delta}_r}{\mu} A_0 \right) b (q^{\mu,\delta}_r)\, dr\right]^t_0\\
\notag &\quad -\int_0^t
\exp \left(\frac{\beta^{\mu,\delta}_s}{\mu} A_0 \right) \left(\frac{1}{\alpha(q^{\mu,\delta}_s)}\right)'A^{-1}_0\int_0^s
\exp \left(-\frac{\beta^{\mu,\delta}_r}{\mu} A_0 \right) b (q^{\mu,\delta}_r)\, dr
\, ds-\int_0^t \frac{1}{\alpha(q^{\mu,\delta}_s)}A^{-1}_0 b
(q^{\mu,\delta}_s)\, ds\\
\notag &=\exp \left(\frac{\beta^{\mu,\delta}_t}{\mu} A_0 \right)
\frac{1}{\alpha(q^{\mu,\delta}_t)}A^{-1}_0 \int_0^t
\exp \left(-\frac{\beta^{\mu,\delta}_s}{\mu} A_0 \right) b (q^{\mu,\delta}_s)\, ds\\
&\notag \qquad\qquad\qquad-\int_0^t
\exp \left(\frac{\beta^{\mu,\delta}_s}{\mu} A_0 \right)\frac{\nabla
\alpha(q^{\mu,\delta}_s)\cdot
p^{\mu,\delta}_s}{\alpha(q^{\mu,\delta}_s)^2}A^{-1}_0\int_0^s
\exp \left(-\frac{\beta^{\mu,\delta}_r}{\mu} A_0 \right) b (q^{\mu,\delta}_r)\, dr
\, ds\\
&\label{eq8} \qquad\qquad\qquad\qquad\qquad\qquad\quad\qquad\qquad\qquad\qquad\qquad\qquad\qquad-\int_0^t \frac{1}{\alpha(q^{\mu,\delta}_s)}A^{-1}_0 b
(q^{\mu,\delta}_s)\, ds.
\end{align}

This yields
\begin{align*}
\notag |I_2(t)|&=\left|\exp \left(\frac{\beta^{\mu,\delta}_t}{\mu} A_0 \right)
\frac{1}{\alpha(q^{\mu,\delta}_t)}A^{-1}_0 \int_0^t
\exp \left(-\frac{\beta^{\mu,\delta}_s}{\mu} A_0 \right) b (q^{\mu,\delta}_s)\,
ds\right.\\
&\notag \qquad \qquad \qquad \quad
\left.-\int_0^t \exp \left(\frac{\beta^{\mu,\delta}_s}{\mu} A_0 \right)\frac{\nabla
\alpha(q^{\mu,\delta}_s)\cdot
p^{\mu,\delta}_s}{\alpha(q^{\mu,\delta}_s)^2}A^{-1}_0\int_0^s
\exp \left(-\frac{\beta^{\mu,\delta}_r}{\mu} A_0 \right) b (q^{\mu,\delta}_r)\, dr
\, ds\right|\\
&\notag \leq \left|\exp \left(\frac{\beta^{\mu,\delta}_t}{\mu} A_0 \right)
\frac{1}{\alpha(q^{\mu,\delta}_t)}A^{-1}_0 \int_0^t
\exp \left(-\frac{\beta^{\mu,\delta}_s}{\mu} A_0 \right) b (q^{\mu,\delta}_s)\,
ds\right|\\
&\notag \qquad \qquad \qquad \quad
+\left|\int_0^t \exp \left(\frac{\beta^{\mu,\delta}_s}{\mu} A_0 \right)\frac{\nabla
\alpha(q^{\mu,\delta}_s)\cdot
p^{\mu,\delta}_s}{\alpha(q^{\mu,\delta}_s)^2}A^{-1}_0\int_0^s
\exp \left(-\frac{\beta^{\mu,\delta}_r}{\mu} A_0 \right) b (q^{\mu,\delta}_r)\, dr
\, ds\right|.
\end{align*}

Considering \eqref{neq05} and Hypothesis~\ref{hyp01},
\begin{align*}
\notag |I_2(t)|&\leq \frac{1}{\alpha_0}\left|\int_0^t
\exp \left(-\frac{\beta^{\mu,\delta}_s}{\mu} A_0 \right) b (q^{\mu,\delta}_s)\,
ds\right|+\int_0^t \frac{|\nabla \alpha|_{\infty}
}{\alpha_0^2}|p^{\mu,\delta}_s|\left|\int_0^s
\exp \left(-\frac{\beta^{\mu,\delta}_r}{\mu} A_0 \right)b (q^{\mu,\delta}_r)\, dr\right|
\, ds.
\end{align*}

Applying Lemma~\ref{th03},
\begin{align*}
\notag |I_2(t)|&\notag \leq
C_1(1+t)^2(1+X_t)\exp \left(\frac{C_2}{\delta}(1+t)(1+X_t)\right)\mu\\
&\qquad\qquad\qquad\qquad\qquad\qquad+\int_0^t|p^{\mu,\delta}_s|C_3(1+s)^2(1+X_s)\exp \left(\frac{C_4}{\delta}(1+s)(1+X_s)\right)\mu
\, ds\\
&\notag \leq
C_1(1+t)^2(1+X_t)\exp \left(\frac{C_2}{\delta}(1+t)(1+X_t)\right)\mu\\
&\qquad\qquad\qquad\qquad\qquad\qquad+C_3(1+t)^2(1+X_t)\exp \left(\frac{C_4}{\delta}(1+t)(1+X_t)\right)\mu\int_0^t
|p^{\mu,\delta}_s|\, ds.
\end{align*}

Note that by Lemma~\ref{th01},
\begin{align*}
\int_0^t |p^{\mu,\delta}_s|\, ds &\leq
C_5(1+t)(1+X_t)\exp \left(\frac{C_6}{\delta}(1+t)(1+X_t)\right) t\\
&\leq
C_5(1+t)^2(1+X_t)\exp \left(\frac{C_6}{\delta}(1+t)(1+X_t)\right)
\end{align*}
and so,
\begin{align*}
\notag |I_2(t)|&\leq
C_1(1+t)^2(1+X_t)\exp \left(\frac{C_2}{\delta}(1+t)(1+X_t)\right)\mu\\
&\qquad\qquad\qquad\qquad\qquad\qquad+C_6(1+t)^4(1+X_t)^2\exp \left(\frac{C_7}{\delta}(1+t)(1+X_t)\right)\mu\\
&\leq C(1+t)^4(1+X_t)^2\exp \left(\frac{C}{\delta}(1+t)(1+X_t)\right)\mu.
\end{align*}

We can apply a similar procedure as in getting the bound for $I_2(t)$ in the case of $I_3(t)$ and get the bound
\begin{equation*}
|I_3(t)|\leq C(1+t)^4(1+X_t)^3\exp \left(\frac{C}{\delta}(1+t)(1+X_t)\right)\mu.
\end{equation*}

Now, we find a bound for $I_4(t)$. From the expression of $I_4(t)$ in \eqref{eq25},
\begin{align*}
\notag |I_4(t)|&=\left|\int_0^t A^{-1}_0
\left(\frac{1}{\alpha(q^{\mu,\delta}_s)}
b(q^{\mu,\delta}_s)-\frac{1}{\alpha(q^{\delta}_s)}b(q^{\delta}_s)\right)ds\right|\\
\notag &=\left|\int_0^t A^{-1}_0
\left(\frac{b(q^{\mu,\delta}_s)\alpha(q^{\delta}_s)-b(q^{\delta}_s)\alpha(q^{\mu,\delta}_s)}{\alpha(q^{\delta}_s)\alpha(q^{\mu,\delta}_s)}
\right)ds\right|\\
\notag &\leq \int_0^t
\left|\frac{b(q^{\mu,\delta}_s)\alpha(q^{\delta}_s)-b(q^{\delta}_s)\alpha(q^{\mu,\delta}_s)}{\alpha(q^{\delta}_s)\alpha(q^{\mu,\delta}_s)}
\right|ds\\ 
&\leq \frac{1}{\alpha_0^2}\int_0^t
|b(q^{\mu,\delta}_s)\alpha(q^{\delta}_s)-b(q^{\delta}_s)\alpha(q^{\delta}_s)|+|b(q^{\delta}_s)\alpha(q^{\delta}_s)-b(q^{\delta}_s)\alpha(q^{\mu,\delta}_s)|ds\\
\notag &\leq \frac{1}{\alpha_0^2}\int_0^t
|\alpha|_{\infty}|b(q^{\mu,\delta}_s)-b(q^{\delta}_s)|+|b|_{\infty}|\alpha(q^{\delta}_s)-\alpha(q^{\mu,\delta}_s)|ds\\
\notag &\leq \frac{1}{\alpha_0^2}\int_0^t
|\alpha|_{\infty} K|q^{\mu,\delta}_s-q^{\delta}_s|+|b|_{\infty}K|q^{\delta}_s-q^{\mu,\delta}_s|ds\\
&\leq C\int_0^t |q^{\mu,\delta}_s-q^{\delta}_s|ds,
\end{align*}
where $K$ is the Lipschitz constant for both $b(q)$ and $\alpha(q)$.

By a similar method, a bound for $I_5(t)$ can be found also. We have
\begin{equation*}
|I_5(t)|\leq \frac{C}{\delta}X_t\int_0^t
|q^{\mu,\delta}_s-q^{\delta}_s|ds.
\end{equation*}

Combining these results and applying the bounds of $I_1(t)$ to $I_5(t)$ to \eqref{eq25}, we obtain
\begin{align*}
|q^{\mu,\delta}_t-q^{\delta}_t|&\leq
C|p_0|(1+t)^2(1+X_t)\exp \left(\frac{C}{\delta}(1+t)(1+X_t)\right)\mu\\
&\qquad\qquad\qquad\qquad+C(1+t)^4(1+X_t)^2\exp \left(\frac{C}{\delta}(1+t)(1+X_t)\right)\mu\\
&\qquad\qquad\qquad\qquad\qquad\qquad+C(1+t)^4(1+X_t)^3\exp \left(\frac{C}{\delta}(1+t)(1+X_t)\right)\mu\\
&\qquad\qquad\qquad\qquad\qquad\qquad\qquad\qquad+C\int_0^t|q^{\mu,\delta}_s-q^{\delta}_s|ds+\frac{C}{\delta}X_t\int_0^t |q^{\mu,\delta}_s-q^{\delta}_s|ds\\
&\leq
C(1+t)^4(1+X_t)^3\exp \left(\frac{C}{\delta}(1+t)(1+X_t)\right)\mu+\frac{C}{\delta}(1+X_t)\int_0^t
|q^{\mu,\delta}_s-q^{\delta}_s|ds.
\end{align*}

Then, from the Gronwall's lemma, we can conclude
\begin{align*}
|q^{\mu,\delta}_t-q^{\delta}_t|&\leq
C(1+t)^4(1+X_t)^3\exp \left(\frac{C}{\delta}(1+t)(1+X_t)\right)\mu
\exp \left(\frac{C}{\delta}(1+X_t)t\right)\\
&\leq
C(1+t)^4(1+X_t)^3\exp \left(\frac{C}{\delta}(1+t)(1+X_t)\right)\mu.
\end{align*}

This gives
\begin{equation*}
\max_{0 \leq t \leq T} |q^{\mu,\delta}_t-q^{\delta}_t|\leq
C(1+T)^4(1+X_T)^3\exp \left(\frac{C}{\delta}(1+T)(1+X_T)\right)\mu,
\end{equation*}

So that, by taking expectation and applying Lemma~\ref{th04},
\begin{align}
\notag E\max_{0 \leq t \leq T} |q^{\mu,\delta}_t-q^{\delta}_t|&\leq
E[C(1+T)^4(1+X_T)^3\exp \left(\frac{C}{\delta}(1+T)(1+X_T)\right)\mu]\\
&\notag \leq C(1+T)^4\mu
E[(1+X_T)^3\exp \left(\frac{C}{\delta}(1+T)(1+X_T)\right)]\\
&\notag \leq C(1+T)^4\mu
(1+T)^{\frac{3}{2}}\exp \left(2(T+1)\frac{C^2}{\delta^2}(1+T)^2+\frac{C}{\delta}(1+T)\right)\\
&\notag \leq C
(1+T)^{\frac{11}{2}}\exp \left(\frac{C}{\delta^2}(1+T)^3\right)\mu.
\end{align}

\end{proof}

\section{Homogenization}

In this section, we consider the case of a fast oscillating periodic magnetic field.\\ 
Consider the solution $q^{\mu,\delta,\epsilon}_t$ of
\begin{equation*}
\left\{ \begin{array}{ll} \mu \, \ddot{q}^{\mu,\delta,\epsilon}_t
=b(q^{\mu,\delta,\epsilon}_t)+\alpha \left({\frac{q^{\mu,\delta,\epsilon}_t}{\epsilon}}\right)A_0 \dot{q}^{\mu,\delta,\epsilon}_t+\dot{w}^{\delta}_t
\\\\
q^{\mu,\delta,\epsilon}_0=q_0  \in \mathbb{R}^2 ,\quad
\dot{q}^{\mu,\delta,\epsilon}_0=p_0 \in \mathbb{R}^2,
\end{array} \right.
\end{equation*}
where $\alpha : \RR^2 \rightarrow \RR$ is a 1-periodic function and $\epsilon >0$ is a constant. By periodicity of $\alpha$, we can consider the domain of $\alpha$ as $\mathbb{T}^2=\RR^2/\mathbb{Z}^2$, the two dimensional unit torus. In this case, a unique weak limit of the process $q^{\mu,\delta,\epsilon}_t$ as  $\mu \downarrow 0$, $\delta \downarrow 0$, and $\epsilon \downarrow 0$ in order exists and we find this limit by applying homogenization results in the literature ( \cite{F01},~\cite{F02},~\cite{FH01},~\cite{BOOKF01},~\cite{PBL01},~\cite{PV01} ) to our system. Note that we solve for $\sigma(q) \equiv I$ for computational convenience. In general, if $\sigma(q)\sigma(q)^*$ is positive definite for all $q \in \RR^2$, we can find a weak limit. For the proof of homogenization results, we need more restrictive assumptions than Hypothesis~\ref{hyp01}.

\begin{hyp}\label{hyp02}\ 
\begin{enumerate}
\item $b : \RR^2 \rightarrow \RR^2$ is twice continuously differentiable and bounded with its derivatives.
\item $\alpha : \RR^2 \rightarrow \RR$ is twice continuously differentiable and bounded with its derivatives. Moreover, 
\begin{equation*}
\inf_{q \in \RR^2} \alpha(q) = \alpha_0>0.
\end{equation*}
\end{enumerate}
\end{hyp}

\begin{prop} Under Hypothesis~\ref{hyp02}, $q^{\mu,\delta,\epsilon}_t$ converges to $\hat{q}_t$ weakly as $\mu \downarrow 0$, $\delta \downarrow 0$, and $\epsilon \downarrow 0$ in order, where $\hat{q}_t$ solves
\begin{equation*}
\left\{ \begin{array}{ll} \dot{\hat{q}}_t
=\hat{b}(\hat{q}_t)+\hat{\sigma}\dot{w}_t
\\\\
\hat{q}_0=q_0  \in \mathbb{R}^2.
\end{array} \right.
\end{equation*}
Here,
\[
\hat{b}(q)=\left(\frac{1}{\int_{\mathbb{T}^2}\alpha(q)dq}\int_{\mathbb{T}^2}\left( I-D\chi(q)\right)dq\,A_0\right)b(q)
\]
and
\[
\hat{\sigma}\hat{\sigma}^*=\frac{1}{\int_{\mathbb{T}^2}\alpha(q)dq}\int_{\mathbb{T}^2}\frac{1}{\alpha(q)}\left( I-D\chi(q)\right)\left( I-D\chi(q)\right)^*dq
\]
with $\chi(q)=(\chi^1(q),\chi^2(q))$ solving
\[
L \chi^i(q)=-\frac{1}{2\alpha^3(q)}\frac{\partial \alpha}{\partial q^i}(q),
\]
where $L$ is the operator
\[
L=\frac{1}{2}\frac{1}{\alpha^2(q)}\Delta_q-\frac{1}{2}\frac{\nabla \alpha(q)}{\alpha^3(q)} \cdot \nabla_q.
\]
\end{prop} 

\begin{proof} 
By Corollary~\ref{th07}, as $\mu \downarrow 0$ first and $\delta \downarrow 0$, $q^{\mu,\delta,\epsilon}_t \rightarrow \hat{q}^{\epsilon}_t$ in probability in $C([0,T];\RR^2)$, where $\hat{q}^{\epsilon}_t$ solves
\begin{equation}\label{eqnn01}
\left\{ \begin{array}{ll} \dot{\hat{q}}^{\epsilon}_t=-\frac{1}{\alpha \left({\frac{\hat{q}^{\epsilon}_t}{\epsilon}}\right)}A_0^{-1} b
(\hat{q}^{\epsilon}_t)-\frac{1}{\alpha\left({\frac{\hat{q}^{\epsilon}_t}{\epsilon}}\right)}A_0^{-1} \circ \dot{w}_t
\\\\
\hat{q}^{\epsilon}_0=q_0  \in \mathbb{R}^2.
\end{array} \right.
\end{equation}

Considering 
\begin{equation*}
A_0^{-1}=-A_0
\end{equation*}
from the definition of $A_0$ in \eqref{neq01}, 
writing \eqref{eqnn01} in It\^{o} integral, we get 
\begin{equation*}
 \dot{\hat{q}}^{\epsilon}_t=\frac{1}{\alpha\left({\frac{\hat{q}^{\epsilon}_t}{\epsilon}}\right)}\tilde{b}
(\hat{q}^{\epsilon}_t)-\frac{1}{2\epsilon}\frac{\nabla \alpha\left({\frac{\hat{q}^{\epsilon}_t}{\epsilon}}\right)}{\alpha^3\left({\frac{\hat{q}^{\epsilon}_t}{\epsilon}}\right)}+\frac{1}{\alpha\left({\frac{\hat{q}^{\epsilon}_t}{\epsilon}}\right)}\dot{\tilde{w}}_t,
\end{equation*}
where 
\begin{equation*}
\tilde{b}(q):=A_0 b(q)
\end{equation*}
and
\begin{equation*}
\tilde{w}_t:=A_0w_t.
\end{equation*}

Note that $\tilde{w}_t$ is also a Wiener process in $\RR^2$.

Under Hypothesis~\ref{hyp02}, we can apply \cite[Theorem 6.1, Chapter 3]{PBL01} to $\hat{q}^{\epsilon}_t$.\\
The normalized solution $m(q)$ of the adjoint equation $L^* m(q)=0$ can be easily found as
\[
m(q)=\frac{1}{\int_{\mathbb{T}^2} \alpha(q) dq} \alpha(q)
\]
and the statement of the proposition follows.
\end{proof}

\bibliographystyle{plain}

\bibliography{Research14}

\end{document}